\documentclass[12pt,amscd]{amsart}
\usepackage[all]{xy}
\usepackage{graphicx}
\usepackage{amsmath,amsxtra,amssymb,latexsym, amscd,amsthm}
\usepackage[pagewise]{lineno}
\usepackage{indentfirst}
\usepackage[mathscr]{eucal}
\usepackage[pagebackref=true]{hyperref}

\newtheorem{thm}{Theorem}[section]

\newtheorem{lem}[thm]{Lemma}
\newtheorem{prop}[thm]{Proposition}
\theoremstyle{definition}
\newtheorem{defn}[thm]{Definition}
\newtheorem{exm}[thm]{Example}

\newtheorem{rem}[thm]{Remark}
\newtheorem{conj}[thm]{Conjecture}
\numberwithin{equation}{section}

\DeclareMathOperator{\NN}{\mathbb {N}}

\DeclareMathOperator{\lk}{lk}

\DeclareMathOperator{\bou}{bou}
\DeclareMathOperator{\dstab}{dstab}
\DeclareMathOperator{\bc}{bc}

\DeclareMathOperator{\depth}{depth}
\DeclareMathOperator{\pd}{pd}
\DeclareMathOperator{\supp}{supp}
\DeclareMathOperator{\sdstab}{sdstab}

\DeclareMathOperator{\charr}{char}

\def\B {\mathcal B}
\def\d {\mathbf d}

\def\a {\mathbf a}
\def\b {\mathbf b}

\def\m {\mathfrak m}

\def\p {\mathfrak p}

\def\k {\mathrm{k}}
\def\h {\widetilde{H}}
\def\c {\mathbf{c}}

\begin{document}

\title[Symbolic depth function of cycles] {Stable value of depth of symbolic 
 powers of edge ideals of graphs}

\author{Nguyen Cong Minh}
\address{Faculty of Mathematics and Informatics, Hanoi University of Science and Technology, 1 Dai Co Viet, Hanoi, Vietnam}
\email{minh.nguyencong@hust.edu.vn}

\author{Tran Nam Trung}
\address{Institute of Mathematics, VAST, 18 Hoang Quoc Viet, Hanoi, Vietnam}
\email{tntrung@math.ac.vn}

\author{Thanh Vu}
\address{Institute of Mathematics, VAST, 18 Hoang Quoc Viet, Hanoi, Vietnam}
\email{vuqthanh@gmail.com}


\date{}

\dedicatory{Dedicated to Professor Ngo Viet Trung on the occasion of his 70th birthday}
\commby{}
\maketitle
\begin{abstract}
    Let $G$ be a simple graph on $n$ vertices. We introduce the notion of bipartite connectivity of $G$, denoted by $\operatorname{bc}(G)$ and prove that 
    $$\lim_{s \to \infty} \operatorname{depth} (S/I(G)^{(s)}) \le \operatorname{bc}(G),$$
    where $I(G)$ denotes the edge ideal of $G$ and $S = \mathrm{k}[x_1, \ldots, x_n]$ is a standard graded polynomial ring over a field $\mathrm{k}$. We further compute the depth of symbolic powers of edge ideals of several classes of graphs, including odd cycles and whisker graphs of complete graphs to illustrate the cases where the above inequality becomes equality.
\end{abstract}

\maketitle

\section{Introduction}
\label{sect_intro}

Let $I$ be a homogeneous ideal in a standard graded polynomial ring $S = \k[x_1,\ldots,x_n]$ over a field $\k$. While the depth function of powers of $I$ is convergent by the result of Brodmann \cite{Br}, the depth function of symbolic powers of $I$ is more exotic. Nguyen and N. V. Trung \cite{NT} proved that for every positive eventually periodic function $f: \NN \to \NN$ there exists an ideal $I$ such that $\depth S/I^{(s)} = f(s)$ for all $s \ge 1$, where $I^{(s)}$ denotes the $s$-th symbolic power of $I$. On the other hand, when $I$ is a squarefree monomial ideal, by the result of Hoa, Kimura, Terai, and T. N. Trung \cite{HKTT} and Varbaro \cite{Va}, 
$$\lim_{s\to \infty} \depth S/I^{(s)}  = \min \left \{ \depth S/I^{(s)} \mid s\ge 1 \right \} = n - \ell_s(I),$$
where $\ell_s(I)$ is the symbolic analytic spread of $I$. Nonetheless, given a squarefree monomial ideal $I$, computing the stable value of depth of symbolic powers of $I$ is a difficult problem even in the case of edge ideals of graphs.

Let us now recall the notion of the edge ideals of graphs. Let $G$ be a simple graph with the vertex set $V(G) = \{1, \ldots, n\}$ and edge set $E(G)$. The edge ideal of $G$, denoted by $I(G)$, is the squarefree monomial ideal generated by $x_ix_j$ where $\{i,j\}$ is an edge of $G$. In \cite{T}, the second author showed that $\lim_{s\to \infty} \depth S/I(G)^s$ equals the number of bipartite connected components of $G$, and that $\depth S/I(G)^s$ stabilizes when it reaches the limit depth. By the results of \cite{NV2, HNTT}, we may assume that $G$ is a connected graph when considering the depth of (symbolic) powers of the edge ideal of $G$. In this case, the result of \cite{T} can be written as 
\begin{equation}\label{eq_limit_depth}
    \lim_{s\to \infty} \depth S/I(G)^s = \begin{cases} 1 & \text{ if } G \text{ is bipartite}\\
0 & \text{ otherwise},\end{cases}
\end{equation}
and the stabilization index of depth of powers of $I(G)$, denoted by $\dstab(I(G))$, is the smallest exponent $s$ such that $\depth S/I(G)^s$ equals the limit depth of powers. Since we expect that the depth functions of symbolic powers of edge ideals are non-increasing, this property should hold for symbolic powers of $I(G)$ as well. In \cite{HLT}, Hien, Lam, and N. V. Trung characterized graphs for which $\lim_{s \to \infty} \depth S/I(G)^{(s)} = 1$ and proved that the stabilization index of depth of symbolic powers in this case is also the smallest exponent $s$ such that $\depth S/I(G)^{(s)} = 1$. For a general non-bipartite graph $G$, we do not know the value $\lim_{s\to \infty} \depth S/I(G)^{(s)}$. 

In this paper, we introduce the notion of bipartite connectivity of $G$ and show that this is tightly connected to the stable value of depth of symbolic powers of $I(G)$. Let $\B(G)$ denote the set of maximal induced bipartite subgraphs $H$ of $G$, i.e., for any $v\in V(G) \setminus V(H)$, the induced subgraph of $G$ on $V(H) \cup \{v\}$ is not bipartite. Note that $H$ might contain isolated vertices. Since $H$ is maximal, it contains at least one edge. Then we define $\bc(G) = \min \{c(H) \mid H \in \B(G)\}$ and call it the bipartite connectivity number of $G$, where $c(H)$ is the number of connected components of $H$.  With this notation, the result of Hien, Lam, and N. V. Trung can be stated as $\lim_{s\to \infty} \depth S/I(G)^{(s)} = 1$ if and only if $\bc(G) = 1$, i.e., there exists an induced connected bipartite subgraph $H$ of $G$ such that $H$ dominates $G$. In this paper, we generalize this result and prove:

\begin{thm}\label{thm_lim_depth_symbolic_power}
    Let $G$ be a simple graph. Then 
    $$\lim_{s\to \infty} \depth S/I(G)^{(s)} \le \bc(G).$$
\end{thm}
In contrast to Eq. \eqref{eq_limit_depth}, we show that the limit depth of symbolic powers of $I(G)$ could be any positive number even when $G$ is a connected graph.

\begin{prop}\label{prop_whisker_1} Let $n \ge 2$ be a positive number and $W_n = W(K_n)$ be the whisker graph on the complete graph on $n$ vertices. Then, $\bc(W_n) = n - 1$ and 
    $$\depth S/I(W_n)^{(s)} = \begin{cases} n & \text{ if } s = 1\\ n-1 & \text{ if } s \ge 2. \end{cases}$$
\end{prop}

We also note that the inequality in Theorem \ref{thm_lim_depth_symbolic_power} could be strict as given in the following example.
\begin{exm}\label{exm_b}
    Let $W$ be the graph obtained by gluing two whiskers at the vertices of a $3$-cycle. Then $\bc(W) = 3$ while 
    $$\depth S/I(W)^{(s)} = \begin{cases} 7 & \text{ if } s = 1\\ 4 & \text{ if } s = 2 \\ 2 & \text{ if } s \ge 3.\end{cases}$$
\end{exm}

Nonetheless, if we cluster the isolated points in a maximal bipartite subgraph $H$ of $G$ by the bouquets in $G$ then we obtain a finer invariant of $G$ that gives the stable value of depth of symbolic powers. More precisely, assume that $H = H_1 \cup \cdots \cup H_c \cup \{p_1, \ldots, p_t\}$ where $H_i$ are connected components of $H$ with at least one edge and $p_1, \ldots, p_t$ are isolated points in $H$. We says that $p_{i_1}, \ldots, p_{i_u}$ are clustered if there exists a $v \in V(G) \setminus V(H)$ such that the induced subgraph of $G$ on $\{v, p_{i_1}, \ldots, p_{i_u}\}$ is a bouquet. Let $\bou_G(H)$ be the smallest number $b$ such that the set $\{p_1, \ldots, p_t\}$ can be clustered into $b$ bouquets in $G$. We call $c'(H) = c + \bou_G(H)$ the number of restricted connected components of $H$. We then define $\bc'(G) = \min \{c'(H)\mid H \in \B(G)\}$ the restricted bipartite connectivity number of $G$. It is easy to see that for the graph $W$ in Example \ref{exm_b}, we have $\bc'(W) = 2$. We conjecture that 

\begin{conj}\label{conj_stable_depth_symbolic}
    Let $G$ be a simple graph. Then 
    $$\lim_{s \to \infty} \depth S/I(G)^{(s)} = \bc'(G).$$
\end{conj}
We verify this conjecture for whisker graphs of complete graphs.

\begin{thm}\label{thm_whisker_complete_graph} Let $\a = (a_1, \ldots, a_n) \in \NN^n$ and $W_{\a}$ be the graph obtained by gluing $a_i$ leaves to the vertex $i$ of a complete graph $K_n$. Assume that $a_i\ge 1$ for all $i = 1, \ldots, n$. Then $\bc'(W_\a) = n - 1$ and 
$$\lim_{s \to \infty } \depth S/I(W_{\a})^{(s)} = n-1.$$    
\end{thm}

Finally, we compute the depth of symbolic powers of edge ideals of odd cycles by extending our argument in \cite{MTV}. This shows that the bound for the index of depth stability of symbolic powers of $I$ given in \cite{HLT} is sharp.
\begin{thm}\label{depth_symbolic_power_cycle}
    Let $I(C_n)$ be the edge ideal of a cycle of length $n = 2k+1 \ge 5$. Then 
    $$\depth S/I(C_n)^{(s)} = \begin{cases} \lceil \frac{n-1}{3} \rceil & \text { if } s = 1\\
    \max (1,\lceil \frac{n-s+1}{3} \rceil) & \text { if } s \ge 2.\end{cases}$$
    In particular, $\sdstab(I(C_n)) = n-2$, where $\sdstab(I)$ is the index of depth stability of symbolic powers of $I$.
\end{thm}

We structure the paper as follows. In Section \ref{sec_pre}, we set up the notation and provide some background. In Section \ref{sec_bound_stable_value}, we prove Theorem \ref{thm_lim_depth_symbolic_power} and compute the depth of symbolic powers of edge ideals of whisker graphs of complete graphs. In Section \ref{sec_depth_symbolic_power}, we prove Theorem \ref{depth_symbolic_power_cycle}.

\section{Preliminaries}\label{sec_pre}

In this section, we recall some definitions and properties concerning depth, graphs and their edge ideals, and the symbolic powers of squarefree monomial ideals. The interested readers are referred to \cite{BH} for more details.

Throughout the paper, we denote by $S = \k[x_1,\ldots, x_n]$ a standard graded polynomial ring over a field $\k$. Let $\m = (x_1,\ldots, x_n)$ be the maximal homogeneous ideal of $S$.

\subsection{Depth} For a finitely generated graded $S$-module $L$, the depth of $L$ is defined to be
$$\depth(L) = \min\{i \mid H_{\m}^i(L) \ne 0\},$$
where $H^{i}_{\m}(L)$ denotes the $i$-th local cohomology module of $L$ with respect to $\m$. We have the following estimates on depth along short exact sequences (see \cite[Proposition 1.2.9]{BH}).

\begin{lem}\label{lem_short_exact_seq}
    Let $0 \to L \to M \to N \to 0$ be a short exact sequence of finitely generated graded $S$-modules. Then 
    \begin{enumerate}
        \item $\depth M \ge \min \{\depth L, \depth N \},$
        \item $\depth L \ge \min \{\depth M, \depth N + 1 \}.$
    \end{enumerate}
\end{lem}

We make repeated use of the following two results in the sequence. The first one is \cite[Corollary 1.3]{R}. The second one is \cite[Theorem 4.3]{CHHKTT}.

\begin{lem}\label{lem_upperbound} Let $I$ be a monomial ideal and $f$ a monomial such that $f \notin I$. Then
    $$\depth S/I \le \depth S/(I:f)$$
\end{lem}

\begin{lem}\label{lem_depth_colon} Let $I$ be a monomial ideal and $f$ a monomial. Then 
$$\depth S/I \in \{\depth (S/I:f), \depth (S/(I,f))\}.$$    
\end{lem}

Finally, we also use the following simple result.

 \begin{lem}\label{lem_depth_variables_clearing} Let $S = \k[x_1,\ldots,x_n]$, $R_1 = \k[x_1,\ldots,x_a]$, and $R_2 = \k[x_{a+1},\ldots,x_n]$ for some natural number $a$ such that $1 \le a < n$. Let $I$ and $J$ be homogeneous ideals of $R_1$ and $R_2$, respectively. Then 
 \begin{enumerate}
     \item $\depth (S/(I+J)) = \depth (R_1/I) + \depth (R_2/J)$.
     \item Let $P = I + (x_{a+1}, \ldots, x_b)$. Then $ \depth (S/P) = \depth (R_1/I) + (n-b).$
 \end{enumerate}
 \end{lem} 
\begin{proof}
Part (1) is standard; for example, see \cite[Lemma 2.3]{NV2}.

Part (2) follows from Part (1) and the fact that $\depth (R_2/(x_{a+1},\ldots,x_b)) = (n-b).$
\end{proof}

\subsection{Depth of Stanley-Reisner rings} Let $\Delta$ be a simplicial complex on the vertex set $V(\Delta) = [n] = \{1,\ldots,n\}$. For a face $F \in \Delta$, the link of $F$ in $\Delta$ is the subsimplicial complex of $\Delta$ defined by 
$$\lk_{\Delta}F=\{G\in\Delta \mid  F\cup G\in\Delta, F\cap G=\emptyset\}.$$

For each subset $F$ of $[n]$, let $x_F=\prod_{i\in F}x_i$ be a squarefree monomial in $S$. We now recall the Stanley-Reisner correspondence.

\begin{defn}For a squarefree monomial ideal $I$, the Stanley-Reisner complex of $I$ is defined by
$$ \Delta(I) = \{ F \subseteq [n] \mid x_F \notin I\}.$$

For a simplicial complex $\Delta$, the Stanley-Reisner ideal of $\Delta$ is defined by
$$I_\Delta = (x_F \mid  F \notin \Delta).$$
The Stanley-Reisner ring of $\Delta$ is $\k[\Delta] =  S/I_\Delta.$
\end{defn}

\begin{defn} The $q$-th reduced homology group of $\Delta$ with coefficients over $\k$, denoted $\h_q(\Delta; \k)$ is defined to be the $q$-th homology group of the augmented oriented chain complex of $\Delta$ over $\k$.
\end{defn}

From the Hochster's formula, we deduce the following:
\begin{lem}\label{lem_depth_formula} Let $\Delta$ be a simplicial complex. Then 
$$\depth (\k[\Delta]) = \min \{ |F| + i \mid \h_{i - 1} (\lk_\Delta F;\k) \neq 0 \mid F \in \Delta\}.$$    
\end{lem}
\begin{proof}
    By definition, $\depth (k[\Delta]) = \min \{ i \mid H^i_\m (k[\Delta]) \neq 0\}$. By Hochster's formula \cite[Theorem 5.3.8]{BH}, the conclusion follows.
\end{proof}
We will also use the following Nerve Theorem of Borsuk \cite{B}. First, we recall the definition of the Nerve complex. Assume that the set of maximal facets of $\Delta$ is $A = \{A_1,\ldots,A_r\}$. The Nerve complex of $\Delta$, denoted by $N(\Delta)$ is the simplicial complex on the vertex set $[r] = \{1,\ldots,r\}$ such that $F \subseteq [r]$ is a face of $N(\Delta)$ if and only if 
$$\cap_{j\in F} A_j \neq \emptyset.$$

\begin{thm}\label{thm_nerve} Let $\Delta$ be a simplicial complex. Then for all integer $i$, we have 
$$\h_i (N(\Delta);\k) \cong \h_i(\Delta;\k).$$   
\end{thm}

\subsection{Graphs and their edge ideals} 

Let $G$ denote a finite simple graph over the vertex set $V(G)=[n] = \{1,2,\ldots,n\}$ and the edge set $E(G)$. The edge ideal of $G$ is defined to be
$$I(G)=(x_ix_j~|~\{i,j\}\in E(G))\subseteq S.$$
For simplicity, we often write $i \in G$ (resp. $ij \in G$) instead of $i \in V(G)$ (resp. $\{i,j\} \in E(G)$). By abuse of notation, we also call $x_i$ a vertex of $G$ and $x_i x_j \in I(G)$ an edge of $G$.

A path $P_n$ of length $n -1$ is the graph on $[n]$ whose edges are $\{i,i+1\}$ for $i = 1, \ldots,n-1$. A cycle $C_n$ of length $n \ge 3$ is the graph on $[n]$ whose edges are $\{i,i+1\}$ for $i = 1, \ldots, n-1$ and $\{1,n\}$. 

A clique in $G$ is a complete subgraph of $G$ of size at least $2$.

A graph $H$ on $[n]$ is called bipartite if there exists a partition $[n] = X \cup Y$, $X \cap Y = \emptyset$ such that $E(H) \subseteq X \times Y$. When $E(H) = X \times Y$, $H$ is called a complete bipartite graph, denoted by $K_{X, Y}$. A bouquet is a complete bipartite graph with $|X| = 1$. 

For a vertex $x\in V(G)$, let the neighbourhood of $x$ be the subset $N_G(x)=\{y\in V(G) \mid \{x,y\}\in E(G)\}$, and set $N_G[x]=N_G(x)\cup\{x\}$. The degree of a vertex $x$, denoted by $\deg_G(x)$ is the number of neighbors of $x$. A leaf is a vertex of degree $1$. The unique edge attached to a leaf is called a leaf edge. Denote $d_G(x)$ the number of non-leaf edges incident to $x$.

\subsection{Projective dimension of edge ideals of weakly chordal graphs} A graph $G$ is called weakly chordal if $G$ and its complement do not contain an induced cycle of length at least $5$. The projective dimension of edge ideals of wealy chordal graphs can be computed via the notion of strongly disjoint families of complete bipartite subgraphs, introduced by Kimura \cite{K}. For a graph G, we consider all families of (non-induced) subgraphs $B_1, \ldots, B_g$ of $G$ such that
\begin{enumerate}
    \item each $B_i$ is a complete bipartite graph for $1 \le i \le g$,
    \item the graphs $B_1, \ldots, B_g$ have pairwise disjoint vertex sets,
    \item there exist an induced matching $e_1,\ldots, e_g$ of $G$ for each $e_i \in E(B_i)$ for $1\le i \le g$.
\end{enumerate}
Such a family is termed a strongly disjoint family of complete bipartite subgraphs. We define 
$$ d(G) = \max \left ( \sum_1^g |V (B_i)| - g \right ),$$
where the maximum is taken over all the strongly disjoint families of complete bipartite subgraphs $B_1, \ldots, B_g$ of $G$. We have the following result of Nguyen and Vu \cite[Theorem 7.7]{NV1}.

\begin{thm}\label{thm_pd_weakly_chordal}
    Let $G$ be a weakly chordal graph with at least one edge. Then 
    $$\pd (S/I(G)) = d(G).$$
\end{thm}

We now use it to compute the depth of the edge ideals of whisker graphs of complete graphs.

\begin{lem}\label{lem_depth_whisker} Let $\a = (a_1,\ldots,a_n) \in \NN^n$ and $W_\a$ be the graph obtained by gluing $a_i$ leaves to the vertex $i$ of a complete graph $K_n$. Assume that $a_1 \ge \ldots \ge a_k > 0 = a_{k+1} = \cdots = a_n$. Then 
$$\depth (S/I(W_\a)) = 1 + a_2 + \cdots + a_k.$$    
\end{lem}
\begin{proof}
    From the definition of $W_\a$, it is clear that $W_\a$ is a chordal graph. For simplicity of notation, we assume that 
    \begin{align*}
        V(G) &= \{x_1,\ldots,x_n\} \cup \{y_{i,j} \mid i = 1, \ldots,k, j = 1, \ldots,a_i\} \\
        E(G) &= \{\{x_i, x_j\} \mid i\neq j \in [n]\} \cup \{ \{x_i, y_{i,j}\} \mid i = 1,\ldots,k, j = 1,\ldots,a_i\}.
    \end{align*}
   For any edges $e_1,e_2$ of $W_\a$, we have $N_{W_\a}[e_1] \cap e_2 \neq \emptyset$. Hence, the induced matching number of $W_\a$ is $1$. Now, let $B$ be a complete bipartite subgraph of $W_\a$ with bipartition $V(B) = U_1 \cup U_2$. Let $X = \{x_1,\ldots,x_n\}$ and $Y = \{y_{i,j} \mid i = 1,\ldots,k,j = 1,\ldots,a_i\}$. If $V(B) \cap Y = \emptyset$ then $|V(B)| \le n$. Now, assume that $y_{i,j} \in U_1$ for some $i,j$. Then $x_i \in U_2$ and $y_{k,l} \notin V(B)$ for any $k \neq i$ since $B$ is a complete bipartite graph. Hence, $|V(B)| \le n + a_i$. Therefore, for any complete bipartite subgraph $B$ of $W_\a$, we have 
   $$|V(B) | \le n + \max \{a_i \mid i = 1, \ldots,n\} = n + a_1.$$
   Furthermore, let $U_1 = \{x_1\}$, $U_2 = \{x_2,\ldots,x_n,y_{1,1},\ldots,y_{1,a_1} \}$ and $B = K_{U_1,U_2}$ then $B$ is a complete bipartite subgraph of $W_\a$ with $|V(B)| = n + a_1$. By Theorem \ref{thm_pd_weakly_chordal}, we deduce that 
   $$\pd(S/I(W_\a)) = n + a_1 -1.$$
   The conclusion follows from the Auslander-Buchsbaum formula.
\end{proof}

\subsection{Symbolic powers of edge ideals}
Let $I$ be a squarefree monomial ideal in $S$ with the irreducible decomposition
$$I = \p_1 \cap \cdots \cap \p_m.$$
The $s$-th symbolic power of $I$ is defined by 
$$I^{(s)} = \p_1^s \cap \cdots \cap \p_m^s.$$

By the proof of \cite[Theorem 5.2]{KTY}, we have 
\begin{lem}\label{lem_colon_leaf} Assume that $e$ is a leaf edge of $G$. Then for all $s \ge 2$ we have $I(G)^{(s)} : e = I(G)^{(s-1)}$. In particular, $\depth S/I(G)^{(s)}$ is a non-increasing function.    
\end{lem}

We also have the following simple result that will be used later.

\begin{lem}\label{lem_colon_complete_graphs} Assume that $n \ge 2$ be an integer. Let $K_n$ be the complete graph on $n$ vertices. Then 
$I(K_n)^{(n)} : (x_1\cdots x_n)  = I (K_n).$
\end{lem}
\begin{proof} For each $i = 1, \ldots, n$, let $\p_i = (x_1,\ldots,x_{i-1},x_{i+1},\ldots,x_n)$. Then, $I(K_n) = \p_1 \cap \cdots \cap \p_n$. Since $x_i \notin \p_i$, we deduce that $\p_i^n : (x_1\cdots x_n) = \p_i$. Hence, 
\begin{align*}
    I(K_n)^{(n)} : (x_1 \cdots x_n) &= (\p_1^n \cap \cdots \cap \p_n^n) : (x_1 \cdots x_n) \\
    & = \left ( \p_1^n : (x_1 \cdots x_n) \right ) \cap \cdots \cap \left (\p_n^n: (x_1 \cdots x_n) \right ) \\
    & = \p_1 \cap \cdots \cap \p_n = I(K_n).    
\end{align*}
The conclusion follows.
\end{proof}

\section{Stable value of depth of symbolic powers of edge ideals}\label{sec_bound_stable_value}
In this section, we prove that the stable value of depth of symbolic powers of edge ideals is at most the bipartite connectivity number of $G$. We assume that $S = \k[x_1, \ldots, x_n]$ and $G$ is a simple graph on $V(G) = \{1,\ldots,n\}$. For an exponent $\a = (a_1, \ldots, a_n) \in \NN^n$, we set $x^\a = x_1^{a_1} \cdots x_n^{a_n}$ and $|\a| = a_1 + \cdots + a_n$.

We first introduce some notation. Let $H$ be a connected bipartite graph with the partition $V(H) = X \cup Y$. The bipartite completion of $H$, denoted by $\tilde H$ is the complete bipartite graph $K_{X,Y}$. Now, assume that $H = H_1 \cup \cdots \cup H_c \cup \{p_1, \ldots, p_t\}$ where $H_1, \ldots, H_c$ are connected components of $H$ with at least one edge, and $p_1, \ldots, p_t$ are isolated points of $H$. Then the bipartite completion of $H$ is defined by $\tilde H = \tilde H_1 \cup \cdots \cup \tilde H_c \cup \{p_1, \ldots, p_t\}$. We have 

\begin{lem}\label{lem_colon_bipartite} Let $H$ be a bipartite graph. Let $\a = \d(H) = (d_H(1), \ldots, d_H(n)) \in \NN^n$ and $s = |\a| / 2$. Then 
$$\sqrt{I(H)^{s+1} : x^\a} = I({\tilde H}),$$
where $\tilde H$ is the bipartite completion of $H$.     
\end{lem}
\begin{proof} Since variables corresponding to isolated points do not appear in $I(H)$, we may assume that $H$ does not have isolated points. Assume that $H = H_1 \cup \cdots \cup H_c$ where $H_i$ are connected components of $H$ with at least one edge. Let $\a_i = \d(H_i)$. Note that $x^{\a_i}$ is equal to the product of non-leaf edges of $H_i$, hence $|\a_i|$ is even for all $i$. Let $s_i = |\a_i| / 2$. Now assume that $f \in \sqrt{I(H)^{s+1}: x^\a}$ with $f = f_1 \cdots f_c$ and $\supp f_i \subseteq V(H_i)$. Then we have $f^m x^\a \in I(H)^{s+1}$ for some $m > 0$. Thus, we must have $f_i^m x^{\a_i} \in I(H_i)^{s_i + 1}$ for some $i$. Hence, we may assume that $H$ is connected. The conclusion then follows from \cite[Lemma 3.1]{T} and \cite[Lemma 2.19]{MNPTV}.
\end{proof}

Now, let $H$ be a maximal induced bipartite subgraph of $G$, i.e., for any $v\in V(G) \setminus V(H)$ the induced subgraph of $G$ on $V(H) \cup \{v\}$ is not bipartite. In particular, $H$ contains at least one edge. Assume that $H = H_1 \cup \cdots \cup H_c \cup \{p_1, \ldots, p_t\}$ where $H_i$ are connected components of $H$ with at least one edge and $p_1, \ldots, p_t$ are isolated points of $H$. Then $c(H) = c+ t$ is the number of connected components of $H$. We have
\begin{lem}\label{lem_upperbound_stable_value}
    Let $H$ be a maximal induced bipartite subgraph of $G$. Then 
        $$\depth (S/(I(G)^{(s)}) \le c(H),$$
        for all $s \ge |E(H)|+1$, where $c(H)$ is the number of connected components of $H$.
\end{lem}
\begin{proof} Assume that $H = H_1 \cup \cdots \cup H_c \cup \{p_1, \ldots, p_t\}$ where $H_1, \ldots, H_c$ are connected components of $H$ with at least one edge and $p_1,\ldots, p_t$ are isolated points of $H$. Let $\b = \d(H)$ and $x^\a = x^\b \cdot \prod (e \mid e \text{ is a leaf edge of } H)$. Then $x^\a$ is the product of edges of $H$. Let $s = |\a| /2 = |E(H)|$. By \cite[Corollary 2.7]{MNPTV}, $x^\a \notin I(G)^{(s+1)}$. We claim that 
\begin{equation}\label{eq_3_1}
    \sqrt{I(G)^{(s+1)} : x^\a} = I(\tilde H ) + (x_j \mid j \in V(G) \setminus V(H)).
\end{equation}
By Lemma \ref{lem_colon_bipartite}, it suffices to prove that $x_j \in \sqrt{I(G)^{(s+1)}: x^\a}$ for all $j \in V(G) \setminus V(H)$. Since the induced subgraph of $G$ on $\{j\} \cup H$ is not bipartite, there must exist a connected component, say $H_1$ of $H$ such that the induced subgraph of $G$ on $V(H_1) \cup \{j\}$ has an odd cycle. Let $G_1$ be the induced subgraph of $G$ on $H_1 \cup \{j\}$. Let $j, 1,\ldots, 2k$ be an induced odd cycle in $G_1$. Then we have $x_jx_1\cdots x_{2k} \in I(G_1)^{(k+1)}$. Furthermore, $x_1 \cdots x_{2k} = \prod_{j=1}^k e_j$ is a product of $k$ edges of $H_1$. By the definition of $\a$, we have $x^{\a_1}$ equals the products of all edges of $H_1$. In other words, $x^{\a_1} = x_1 \cdots x_{2k} \cdot h$ with $h \in I(H_1)^{|E(H_1)| - k}.$  Hence, $x_j x^{\a_1} \in I(G_1)^{(s_1 + 1)}$ where $s_1 = |E(H_1)|$. Eq. \eqref{eq_3_1} follows.

By Lemma \ref{lem_upperbound} and Eq. \eqref{eq_3_1}, we deduce that 
$$\depth S/I(G)^{(s+1)} \le \depth S/(I(G)^{(s+1)} : x^\a) \le \depth S/\sqrt{I(G)^{(s+1)} : x^\a }  = c(H).$$
For any $t \ge s+1$, let $x^\c = x^\a \cdot e^{t-s-1}$ where $e$ is an arbitrary edge of $H$. Then we have $x^\c \notin I(G)^{(t)}$ and $\sqrt{I(G)^{(t)} : x^\c} \supseteq \sqrt{I(G)^{(s+1)} : x^\a}$. Hence, $\depth S/I(G)^{(t)} \le c(H)$ for all $t \ge s+1$. The conclusion follows.
\end{proof}

\begin{defn} Let $G$ be a simple graph. Denote by $\B(G)$ the set of all maximal induced bipartite subgraphs of $G$. The bipartite connectivity number of $G$ is defined by 
$$\bc(G) = \min \left \{ c(H) \mid H \in \B(G)   \right \}.$$  
\end{defn}

We are now ready for the proof of Theorem \ref{thm_lim_depth_symbolic_power}. 

\begin{proof}[Proof of Theorem \ref{thm_lim_depth_symbolic_power}]
    The conclusion follows immediately from the definition and Lemma \ref{lem_upperbound_stable_value}.
\end{proof}

We now prove Proposition \ref{prop_whisker_1} giving an example of connected graphs for which the above inequality is equality and that the limit depth of symbolic powers of $I(G)$ could be any positive number.

\begin{proof}[Proof of Proposition \ref{prop_whisker_1}] We may assume that 
\begin{align*}
    V(W_n) & = \{x_1, \ldots, x_n, y_1, \ldots, y_n\} \text{ and } \\
E(W_n) &= \{\{x_i,x_j\}, \{x_i, y_i\} \mid 1 \le i \neq j \le n\}.
\end{align*}

Let $H$ be a maximal bipartite subgraph of $W_n$. Then $y_1, \ldots, y_n \in H$ and $H$ contains at most two vertices in $\{x_1, \ldots, x_n\}$. By the maximality of $H$, we deduce that $H$ must be the induced subgraph of $W_n$ on $\{y_1, \ldots, y_n\} \cup \{x_i, x_j\}$ for some $i \neq j$. Hence, $c(H) = n-1$. Thus, $\bc(W_n) = n-1$.

By Lemma \ref{lem_colon_leaf}, $\depth S/I(W_n)^{(s)}$ is non-increasing. Furthermore, we have 
$$I(W_n)^{(2)} : (x_1x_2) = (x_1y_1,x_2y_2,x_1x_2,y_1y_2,x_3, \ldots, x_n).$$
Hence, $\depth S/I(W_n)^{(2)} \le n - 1$. 

    It remains to prove that $\depth S/I(W_n)^{(s)} \ge n-1$ for all $s \ge 2$. We prove by induction on $n$ and $s$ the following statement. Let $I_k = I(K_n) + (x_1y_1, \ldots, x_ky_k)$ and $S_k = \k[x_1, \ldots, x_n, y_1, \ldots, y_k]$. Then $\depth S_k/I_k^{(s)} \ge k - 1$ for all $2 \le k \le n$ and all $s \ge 1$. 

    Note that $I_k = I(G_k)$ where $G_k = K_n \cup \{ \{ x_i,y_i\} \mid i = 1, \ldots, k\}$. By Lemma \ref{lem_depth_whisker}, $\depth S_k/I_k = k$.

    Since $\m_k$, the maximal homogeneous ideal of $S_k$, is not an associated prime of $I_k$, $\depth S_k / I_k^{(s)} \ge 1$ for all $k$. Thus, we may assume that $s \ge 2$ and $n \ge k \ge 3$. By Lemma \ref{lem_depth_colon},  
$$ \depth S_k/I_k^{(s)} \in \{ \depth (S_k/(I_k^{(s)},x_ky_k)), \depth (S_k/I_k^{(s)}:x_ky_k)\}.$$
By Lemma \ref{lem_colon_leaf}, $I_k^{(s)} : x_k y_k = I_k^{(s-1)}$. Thus, by induction, it suffices to prove that 
$$\depth S_k/(I_k^{(s)},x_ky_k) \ge k - 1.$$
We have $J = (I_k^{(s)},x_k y_k) = (J,x_k) \cap (J,y_k)$. The conclusion follows from induction on $k$ and Lemma \ref{lem_short_exact_seq}.    
\end{proof}

The inequality in Theorem \ref{thm_lim_depth_symbolic_power} might be strict. We will now define a finer invariant of $G$ which we conjecture to be equal to the stable value of depth of symbolic powers of $I(G)$. Let $H = H_1 \cup \cdots \cup H_c \cup \{p_1, \ldots, p_t\}$ be a maximal induced bipartite subgraph of $G$ where $H_1, \ldots, H_c$ are connected components of $H$ with at least one edge and $p_1, \ldots, p_t$ are isolated points. We say that $\{p_{i_1}, \ldots, p_{i_u}\}$ are clustered if there exists $v \in V(G) \setminus V(H)$ such that the induced subgraph of $G$ on $\{v, p_{i_1}, \ldots, p_{i_u}\}$ is a bouquet. Let $\bou_G(H)$ be the smallest number $b$ such that the set $\{p_1, \ldots, p_t\}$ can be clustered into $b$ bouquets in $G$. We call $c'(H) = c + \bou_G(H)$ the number of restricted connected components of $H$. 

\begin{defn} Let $G$ be a simple graph. The restricted bipartite connectivity number of $G$ is defined by
$$\bc'(G) = \min \{c'(H)\mid H \in \B(G)\}.$$    
\end{defn}

We need a preparation lemma to prove Theorem \ref{thm_whisker_complete_graph}.

\begin{lem}\label{lem_colon_whisker} Let $\a = (a_1,\ldots,a_n) \in \NN^n$ be such that $a_i \ge 1$ for all $i = 1, \ldots, n$. Let $W_\a$ be a graph whose vertex set and edge set are 
\begin{align*}
    V(W_\a) &= \{x_1, \ldots,x_n, y_{1,1}, \ldots, y_{1,a_1}, \ldots, y_{n,1}, \ldots, y_{n,a_n} \},\\
    E(W_\a) &= \{ \{x_i, x_j\}, \{x_i, y_{i,\ell}\} \mid \text{ for all } i, j, \ell \text{ such that } 1 \le i \neq j \le n, 1 \le \ell \le a_i \}.
\end{align*}
Then 
$$I(W_\a)^{(n)} : (x_1 \cdots x_n) = I(W_\a) + (y_{1,1},\ldots,y_{1,a_1}) (y_{2,1}, \ldots, y_{2,a_2}) \cdots (y_{n,1}, \ldots, y_{n,a_n}).$$
\end{lem}
\begin{proof} For simplicity of notation, we set $X = \{x_1,\ldots,x_n\}$ and $Y = \{y_{i,j} \mid i = 1,\ldots, n, j = 1,\ldots,a_i\}$. We also denote $I = I(W_\a)$ and 
$$J= I(W_\a) + (y_{1,1},\ldots,y_{1,a_1}) (y_{2,1}, \ldots, y_{2,a_2}) \cdots (y_{n,1}, \ldots, y_{n,a_n}).$$ 
For each $C \subseteq V(W_\a)$, let $m_C = \prod_{x\in C} x$ be a monomial in 
$$S = \k[x_1,\ldots,x_n,y_{1,1},\ldots,y_{1,a_1},\ldots,y_{n,1},\ldots,y_{n,a_n}].$$
Since $W_\a$ is a chordal graph, by \cite[Theorem 3.10]{S}, we have

\begin{equation}\label{eq_Sullivant}
    I^{(n)} = \left ( m_{C_1} \cdots m_{C_t} \mid C_1,\ldots,C_t \text{ are cliques of } W_\a \text{ and } \sum_{i=1}^{t} \left ( |C_i| - 1 \right ) = n \right ).
\end{equation}
Note that the cliques $C_1,\ldots, C_t$ are not necessarily distinct. In $W_\a$, $C \subseteq V(W_\a)$ is a clique if and only if either $C = \{x_i,y_{i,j}\}$ for some $i = 1, \ldots, t$ and $j = 1, \ldots, a_i$ or $C \subseteq X$. In particular, $(x_1\cdots x_n) e \in I^{(n)}$ for all edges $e$ of $W_\a$ and $(x_1 y_{1,j_1}) \cdots (x_n y_{n,j_n}) \in I^{(n)}$ for all $j_1, \ldots, j_n$ such that $1 \le j_\ell \le a_\ell$. Hence, 
\begin{equation}\label{eq_colon_whisker_1}
    J \subseteq I(W_\a)^{(n)} : (x_1 \cdots x_n).
\end{equation}
We now prove by induction on $n$ the reverse containment 
\begin{equation}\label{eq_colon_whisker_2}
    I(W_\a)^{(n)} : (x_1 \cdots x_n) \subseteq J.
\end{equation}
The base case $n = 2$ is clear. Thus, assume that $n \ge 3$. Let $C_1,\ldots,C_t$ be cliques of $W_\a$ such that $\sum_{i=1}^t (|C_i| - 1) = n$. Let $M = m_{C_1} \cdots m_{C_t}$ and $f = x_1 \cdots x_n$. It suffices to prove that $M / \gcd (M, f)  \in J.$ Since $|C_i| \le n$ for all $i = 1, \ldots, t$, we must have $t \ge 2$. We have two cases.

\vspace{1.5mm}
\noindent{\textbf{Case 1.}} $C_i \cap Y = \emptyset$ for all $i= 1, \ldots,t$. In this case, we have $M \in I(K_n)^{(n)}$. By Lemma \ref{lem_colon_complete_graphs}, we deduce that $M/\gcd(M,f) \in I(K_n) \subseteq I(W_\a)$.

\vspace{1.5mm}
\noindent{\textbf{Case 2.}} $C_i \cap Y \neq \emptyset$ for some $i \in \{1,\ldots,t\}$. Since $Y$ is the set of leaves, we deduce that $|C_i| = 2$. For simplicity, we assume that $C_1 = \{x_1,y_{1,1}\}$. If there exists a clique $C_i$ for some $i = 2, \ldots, t$ such that $C_1 \cap C_i \neq \emptyset$, then we must have $x_1 \in C_1 \cap C_i$. In particular, we deduce that $x_1 y_{1,1} \mid M / \gcd(M,f)$. Hence, $M/\gcd(M,f) \in J$. Thus, we may assume that $C_i \cap C_1 = \emptyset$ for all $i = 2, \ldots, t$. In other words, $C_i \subseteq X' \cup Y'$ where $X' = \{x_2,\ldots,x_n\}$ and $Y' =\{ y_{i,j} \mid i = 2,\ldots, n, j = 1, \ldots, a_i\}$. Furthermore, we have $\sum_{i = 2}^t (|C_i| - 1) = n -1$. By Eq. \eqref{eq_Sullivant}, we deduce that $M' = m_{C_2} \cdots m_{C_t} \in I(W_{\a'})^{(n-1)}$, where $\a' = (a_2,\ldots,a_n)$ and $W_{\a'}$ is the whisker graph obtained by gluing $a_i$ leaves to the vertex $i$ of the complete graph on $\{2,\ldots,n\}$. Since $y_{1,1}$ does not divide $f$, we deduce that $y_{1,1} (M'/ \gcd(M',f')) \mid M/\gcd(M,f)$, where $f' = x_2 \cdots x_n$. By induction on $n$, the conclusion follows.
\end{proof}

\begin{proof}[Proof of Theorem \ref{thm_whisker_complete_graph}] We may assume that $a_1 \ge a_2 \ge \cdots \ge a_n \ge 1$. We keep the notations as in Lemma \ref{lem_colon_whisker}.

For ease of reading, we divide the proof into several steps.

\vspace{1.5mm}
\noindent{\textbf{Step 1.}} $\bc'(W_\a) = n-1$. As in the proof of Proposition \ref{prop_whisker_1}, we deduce that a maximal induced bipartite subgraph $H$ of $W_\a$ is an induced subgraph of $W_\a$ on $Y \cup \{x_i,x_j\}$ for some $i \neq j$. For such $H$, we have $c(H) = |\a| - (a_i + a_j) + 1$ but $c'(H) = n-1$ as $\{y_{\ell,1}, \ldots, y_{\ell,a_\ell}\}$ can be clustered into a bouquet in $G$ for all $\ell = 1, \ldots, n$. Thus, $\bc(W_\a) = a_3 + \cdots + a_n+1$ and $\bc'(W_\a) = n-1$.

\vspace{1.5mm}
\noindent{\textbf{Step 2.}} $\depth S/I(W_\a)^{(s)} \ge n-1$ for all $s\ge 1$ and all $\a$ such that $a_i \ge 1$ for $i = 1, \ldots, n$.

First, assume that $s = 1$. By Lemma \ref{lem_depth_whisker}, $\depth S/I(W_\a) = a_2 + \cdots + a_n + 1$. When $a_1 = \cdots = a_n = 1$, the conclusion follows from Proposition \ref{prop_whisker_1}. Thus, we may assume that $s \ge 2$ and $a_1 \ge 2$. By induction, Lemma \ref{lem_depth_colon}, and Lemma \ref{lem_colon_leaf}, it suffices to prove that 
    $$ \depth S/ (I(W_\a)^{(s)},x_1y_{1,a_1}) \ge n-1.$$
    
    Let $J = I(W_\a)^{(s)}$. Then $(J,x_1y_{1,a_1}) = (J,x_1) \cap (J,y_{1,a_1})$. Denote $\a'= (a_2, \ldots, a_n)$ and $W_{\a'}$ the whisker graph obtained by gluing $a_i$ leaves to the vertex $i$ of the complete graph on $\{2, \ldots, n\}$. We have $(J,x_1) = (I(W_{\a'})^{(s)},x_1)$ and $(J,x_1,y_{1,a_1}) = (I(W_{\a'})^{(s)},x_1,y_{1,a_1})$. By Lemma \ref{lem_depth_variables_clearing}, 
    \begin{align*}
        \depth S/(J,x_1) & = a_1 + \depth R/I(W_{\a'})^{(s)}, \\
        \depth S/(J,x_1,y_{1,a_1}) & = a_1 -1 + \depth R/I(W_{\a'})^{(s)},
    \end{align*}
    where $R = \k[x_2, \ldots, x_n, y_{2,1}, \ldots, y_{2,a_2}, \ldots, y_{n,1}, \ldots, y_{n,a_n}]$. By induction, both terms are at least $n - 1$. Finally, $(J,y_{1,a_1}) = (I(W_{\a^{''}})^{(s)},y_{1,a_1})$ where $\a^{''} = (a_1 - 1, a_2, \ldots, a_n)$. Hence, 
    $$\depth S/(J,y_{1,a_1}) = \depth T/I(W_{\a^{''}})^{(s)},$$
    where $T = \k[x_1, \ldots, x_n,y_{1,1}, \ldots, y_{1,a_1-1}, \ldots, y_{n,1}, \ldots, y_{n,a_n}]$. Thus, the conclusion of Step 2 follows from induction and Lemma \ref{lem_short_exact_seq}.

\vspace{1.5mm}
\noindent{\textbf{Step 3.}} $\depth S/I(W_\a)^{(s)} \le n-1$ for all $s \ge n$. 

By Lemma \ref{lem_upperbound} and Lemma \ref{lem_colon_leaf}, it suffices to prove that 
    $$\depth S/ I(W_\a)^{(n)} : (x_1\cdots x_n) \le n-1.$$
Let $J = I(W_\a)^{(n)} : (x_1 \cdots x_n)$. By Lemma \ref{lem_colon_whisker}, we have that 
$$J= I(W_\a) + (y_{1,1},\ldots,y_{1,a_1}) (y_{2,1}, \ldots, y_{2,a_2}) \cdots (y_{n,1}, \ldots, y_{n,a_n}).$$
Therefore, the Stanley-Reisner complex $\Delta(J)$ of $J$ has exactly $n$ facets 
$$F_i = \{ x_i\} \cup \{ y_{j,\ell} \mid j \neq i, \ell = 1, \ldots, a_j\}.$$ 
Hence, $F_1 \cap \cdots \cap F_n = \emptyset$ and for any $j$, we have 
$$F_1 \cap \cdots \cap F_{j-1} \cap F_{j+1} \cap \cdots \cap F_{n} = \{ y_{j,1}, \ldots, y_{j,a_j}\}.$$
Therefore, the Nerve complex of $\Delta(J)$ is isomorphic to the $n-2$-sphere. By Theorem \ref{thm_nerve}, $\h_{n-2} (\Delta(J);\k) \neq 0$. By Lemma \ref{lem_depth_formula}, the conclusion follows.
\end{proof}

\begin{rem} \begin{enumerate}
    \item The notion of maximal bipartite subgraphs of a graph has been studied by many researchers as early as \cite{E, M}. They are interested in finding the maximum number of edges of a maximal bipartite subgraph of $G$. 
    \item In general, the problem of finding a maximum induced bipartite subgraph of a graph is NP-complete \cite{LY}. Nonetheless, we do not know if the problem of computing the bipartite connectivity number or restricted bipartite connectivity number is NP-complete.
\end{enumerate}    
\end{rem}

\begin{rem} \begin{enumerate}
    \item The Cohen-Macaulay property, or depth of the edge ideal of a graph might depend on the characteristic of the base field. For example, consider the following ideal in \cite[Exercise 5.3.31]{Vi}
    \begin{align*}
        I & = (x_1x_3,x_1x_4,x_1x_7,x_1x_{10},x_1x_{11},x_2x_4,x_2x_5,x_2x_8,x_2x_{10},x_2x_{11},\\
        &x_3x_5,x_3x_6,x_3x_8,x_3x_{11},x_4x_6,x_4x_9,x_4x_{11},\\
        &x_5x_7,x_5x_9,x_5x_{11},x_6x_8,x_6x_9,x_7x_9,x_7x_{10},x_8x_{10}).
    \end{align*}
    Then $$\depth S/I = \begin{cases}
        2 & \text{ if } \charr \k = 2, \\
        3 & \text{ otherwise}.
    \end{cases}$$
    But $\depth S/I^{(s)} = 1$ for all $s \ge 2$, regardless of the characteristic of the base field $\k$.
    \item By the result of the second author \cite{T}, the stable value of depth of powers of edge ideals of graphs does not depend on the characteristic of the base field $\k$. If Conjecture \ref{conj_stable_depth_symbolic} holds, the stable value of depth of symbolic powers of edge ideals also does not depend on the characteristic of the base field $\k$. This is in contrast to the asymptotic behavior of the regularity of (symbolic) powers of edge ideals as \cite[Corollary 5.3]{MV} shows that the linearity constant of the regularity function of (symbolic) powers of edge ideals of graphs might depend on the characteristic of the base field $\k$.
\end{enumerate}
\end{rem}
\section{Depth of symbolic powers of edge ideals of cycles}\label{sec_depth_symbolic_power}

In this section, we compute the depth of symbolic powers of edge ideals of cycles. The purpose of this is twofold. First, together with Proposition \ref{prop_whisker_1}, this gives the first classes of non-bipartite graphs where one computes explicitly the depth of symbolic powers of their edge ideals. Second, this shows that the stabilization index of depth of symbolic powers of $G$ is tightly connected to the stabilization index of depth of powers of maximal induced bipartite subgraphs of $G$. 

We fix the following notation. Let $S = \k[x_1,\ldots,x_n]$ and $C_n$ be a cycle of length $n$. For each $i = 1, \ldots, n-1$, we denote $e_i = x_ix_{i+1}$. Let $\varphi(n,t) = \lceil \frac{n-t+1}{3} \rceil$. We recall the following results (Lemma 3.4, Lemma 3.10, Lemma 3.11, and Theorem 1.1) from \cite{MTV}.

\begin{lem}\label{lem_lowerbound_path} Let $H$ be any subgraph of $P_n$. Then, for any positive integer $t$ with $t < n$, we have that
$$ \depth \left ( S/ (I(P_n)^t + I(H)) \right ) \ge \varphi(n,t).$$    
\end{lem}

\begin{lem}\label{lem_lowerbound_cycle_aux} Let $H$ be a non-empty subgraph of $C_n$. Then for $t \ge 2$, we have that 
    $$\depth (S/(I(C_n)^t + I(H))) \ge \varphi(n,t).$$
\end{lem}

\begin{lem}\label{lem_upperbound_cycle} Assume that $I = I(C_n)$ and $t \le n-2$. Then 
$$\depth (S/(I^{t} : (e_2 \cdots e_t))) \le \varphi(n,t).$$
\end{lem}

\begin{thm}\label{depth_power_cycle}
    Let $I(C_n)$ be the edge ideal of a cycle of length $n \ge 5$. Then 
    $$\depth \left ( S/I(C_n)^t \right ) = \begin{cases} \left \lceil \frac{n-1}{3} \right \rceil, & \text { if } t = 1,\\
    \left \lceil \frac{n-t+1}{3} \right \rceil, & \text { if } 2 \le t < \left \lceil \frac{n+1}{2} \right  \rceil,\\
    1, & \text{ if } n \text{ is even and } t \ge  \frac{n}{2} + 1,\\
    0, & \text{ if } n \text{ is odd and } t \ge \frac{n+1}{2}.\end{cases}$$
\end{thm}

Now, assume that $n = 2k+1$ where $k \ge 2$ is a positive integer. For a positive integer $s \in \NN$, we write $s = a(k+1) + b$ for some $a, b \in \NN$ and $0 \le b \le k$. Let $f = x_1 \cdots x_n$. By \cite[Theorem 3.4]{GHOS}, we have 
\begin{equation}\label{eq_sym_cycles_expansion}
    I(C_n)^{(s)} = \sum_{j =0}^a I(C_n)^{s - j (k+1)} f^j.
\end{equation}

We now establish some preparation results.

\begin{lem} \label{lem_upperbound_depth_symbolic_power}
    Assume that $I = I(C_n)$, $e_i = x_ix_{i+1}$ for all $i = 1, \ldots,n-1$. Then for all $s \le n - 2$, we have 
    $$\depth S/I^{(s)} \le \depth S/(I^{(s)} : e_2 \cdots e_{s-1}) \le \varphi(n,s).$$
\end{lem}
\begin{proof} Let $f = x_1 \cdots x_n$. By Eq. \eqref{eq_sym_cycles_expansion}, we have that $I^{(s)} = I^s$ when $s \le k$. Now, assume that $k+1 \le s \le n-2 = 2k-1$. By Eq. \eqref{eq_sym_cycles_expansion}, we have that 
$$I^{(s)} = I^s + f I^{s - k -1}.$$
Since $f / \gcd(f, e_2 \cdots e_{t-1}) \in I \subseteq I^s : (e_2 \cdots e_{s-1})$, we deduce that 
$$I^{(s)} : (e_2 \cdots e_{t-1}) = I^s : (e_2 \cdots e_{s-1}).$$
The conclusion follows from Lemma \ref{lem_upperbound_cycle}.
\end{proof}

\begin{lem}\label{lem_colon_sym_powers} Let $f = x_1 \cdots x_n$. Then for all integer $s$ such that $k + 1 \le s \le n-2$, 
$$I^{(s)} : f = I^{s - k - 1}.$$    
\end{lem}
\begin{proof} Let $\p_1, \ldots, \p_t$ be the associated primes of $I$. Then 
$$I^{(s)} = \p_1^s \cap \cdots \cap \p_t^s.$$
Since $\p_i$ is generated by $k+1$ variables for all $i = 1, \ldots, t$, we have $\p_i^s : f = \p_i^{s - k -1}$. Hence, $I^{(s)} : f = I^{(s-k-1)} = I^{s-k-1}$ since $s \le 2k-1$.
\end{proof}

We are now ready for the proof of Theorem \ref{depth_symbolic_power_cycle}. 

\begin{proof}[Proof of Theorem \ref{depth_symbolic_power_cycle}] By Eq. \eqref{eq_sym_cycles_expansion} and Theorem \ref{depth_power_cycle}, it remains to consider the cases where $k + 1 \le s \le 2k-1$. Let $f = x_1 \cdots x_n$. By Lemma \ref{lem_depth_colon}, Lemma \ref{lem_upperbound_depth_symbolic_power}, Lemma \ref{lem_colon_sym_powers}, and Theorem \ref{depth_power_cycle}, it suffices to prove that 
    \begin{equation}\label{eq_4_1}
        \depth (S/ (I^{(s)} + f)) \ge \varphi(n,s).
    \end{equation}
    Write $f = e_1 f_1$ where $f_1 = x_3 \cdots x_n$. We have $I^{(s)} + f = (I^{(s)},e_1) \cap (I^{(s)},f_1)$. For each $i = 1, \ldots, k-1$, we can write $f_i = e_{2i+1} f_{i+1}$. By repeated use of Lemma \ref{lem_short_exact_seq} and the fact that for any subgraph $H$ of $C_n$ we have 
    $$I^{(s)} + I(H) + f_i = (I^{(s)} + I(H) + (e_{2i+1})) \cap (I^{(s)} + I (H) + (f_{i+1})),$$ 
    it suffices to prove the following two claims

\vspace{1.5mm}
\noindent{\textbf{Claim 1.}} For any non-empty subgraph $H$ of $C_n$, we have $$\depth S/(I^{(s)} + I(H)) \ge \varphi(n,s).$$

\vspace{1.5mm}
\noindent{\textbf{Claim 2.}} For any (possibly empty) subgraph $H$ of $C_n$, we have $$\depth (S/(I^{(s)} + I(H) + (x_{n-2}x_{n-1}x_n))) \ge \varphi(n,s).$$

\vspace{1.5mm}
\noindent{\textbf{Proof of Claim 1.}} Since $k + 1 \le s \le 2k-1$, by Eq. \eqref{eq_sym_cycles_expansion}, we have that 
$$I^{(s)} = I^s + f I^{s- k -1}.$$
For any non-empty subgraph $H$ of $C_n$, we have $f \in I(H)$. Therefore, $I^{(s)} + I(H) = I^s + I(H)$. The conclusion follows from Lemma \ref{lem_lowerbound_cycle_aux}.

\vspace{1.5mm}
\noindent{\textbf{Proof of Claim 2.}} Let $J = I^{(s)} + I(H) + (x_{n-2}x_{n-1}x_n)$ and $e = x_{n-2}x_{n-1}$. Note that $J + (e)$ can be expressed as $I^{(s)} + I(H_1)$ for some subgraph $H_1$ of $C_n$ and $J:e = I(P_{n-1})^{s-1} + I(H') + ( x_n)$ where $H'$ is a subgraph of $P_{n-1}$. The claim follows from Lemma \ref{lem_depth_colon}, Claim 1, and Lemma \ref{lem_lowerbound_path}. The conclusion follows.
\end{proof}

\begin{rem}  
For cycles $C_{2k}$ of even length, by the result of Simis, Vasconcelos, and Villarreal \cite{SVV}, $I(C_{2k})^{(s)} = I(C_{2k})^s$ for all $s \ge 1$. The depth of powers of the edge ideal of $C_{2k}$ has been computed in \cite[Theorem 1.1]{MTV}.
\end{rem}

\section*{Acknowledgments}
Tran Nam Trung is partially supported by Project NCXS02.01/22-23 of the Vietnam Academy of Science and Technology. We are grateful to an anonymous referee for his/her thoughtful suggestions and comments to improve the readability of our manuscript.

\end{document}